\documentclass[12pt]{article}
\usepackage{amsmath}
\usepackage{amsthm}
\usepackage{amssymb}
\everymath{\displaystyle}
\begin{document}
\newcommand{\RR}{\ensuremath{{\mathbb R}}}
\newcommand{\ee}{\mathrm{e}}
\newcommand{\ACloc}{\mathrm{AC}_{\mathrm{loc}}}
\newtheorem{theorem}{Theorem}
\newtheorem{lemma}{Lemma}
\newtheorem{corollary}{Corollary}
\newtheorem{remark}{Remark}
\title{The order of Lebesgue constant of Lagrange interpolation on several intervals}
\author{A.~L.~Lukashov\thanks{Research is supported by RFBR-TUBITAK project No.14-01-91370/113F369}\\
Faith University\\
34500 Buyukcekmece Istanbul (Turkey)\\
E-mail: alexey.lukashov@gmail.com\\
and\\
Saratov State University\\
410012 Astrakhanskaya, 83 Saratov (Russian Federation)\\
\and
J.~Szabados\\
Alfr\'ed R\'enyi Institute of Mathematics\\
P.O.B. 127, H-1364 Budapest (Hungary)\\
E-mail: szabados.jozsef@renyi.mta.hu}
\date{}
\maketitle
\begin{abstract}
We consider Lagrange interpolation on the set of finitely many intervals. This problem is closely related to the least deviating polynomial from zero on such sets. We  will obtain lower and upper estimates for the corresponding Lebesgue constant. The case of two intervals of equal lengths is simpler, and an explicit construction for  two non-symmetric intervals will be given  only in a special case.

\end{abstract}

\section{Lower estimate}

Let

\begin{equation*}
-1=a_0<b_0<a_1<b_1<\dots<a_s<b_s=1
\end{equation*}
be a finite partition of the interval $[-1,1]$, and let
\begin{equation}
\label{set}
E_s:=\bigcup_{i=0}^s [a_i,b_i]
\end{equation}
be the corresponding set of pairwise disjoint intervals. Consider the Lagrange interpolation on the nodes
\begin{equation}
\label{nodes}
(-1\le)x_n<x_{n-1}<\cdots<x_1(\le1),\qquad x_k=x_{k,n}\in E_s,\quad k=1,2,\dots,n\,,
\end{equation}
and let $\omega_n(x):=\prod_{k=1}^n(x-x_k)$. The Lebesgue function of interpolation is defined as
\begin{equation*}
\lambda_n(x,\omega_n,E_s):=\sum_{k=1}^n|\ell_k(x)|\,,
\end{equation*}
where
\begin{equation*}
\ell_k(x):=\frac{\omega_n(x)}{\omega_n'(x_k)(x-x_k)}\,.
\end{equation*}
Our purpose is to estimate the Lebesgue constant
\begin{equation*}
\lambda_n(\omega_n,E_s):=\Vert\lambda_n(x,\omega_n,E_s)\Vert_{E_s}\,,
\end{equation*}
where $\Vert\cdot\Vert_{(\cdot)}$ is the supremum norm of the function over the set indicated.

The classic result of Faber \cite{faber} says that for any set of nodes (\ref{nodes}) we have
\begin{equation*}
\lambda_n(\omega_n,E_0)\ge c\log n
\end{equation*}
with some absolute constant $c>0$. Our first result states that this holds in the more general situation (\ref{set}) as well.

\begin{theorem}
\label{thm1}
For any system of nodes {\rm(\ref{nodes})} we have
\begin{equation*}
\lambda_n(\omega_n,E_s)\ge c(E_s)\log n
\end{equation*}
with a constant $c(E_s)>0$ depending on the partition {\rm(\ref{set})}.
\end{theorem}

\begin{proof} The statement sounds like an obvious consequence of the quoted result of Faber, but in fact we need a strengthening of that result. Consider our interpolation process on the whole interval $[-1,1]$, and apply the deep result of Erd\H os and V\'ertesi \cite{ev} which says that for any system of nodes on $E_0$ and any $\varepsilon>0$, there exists a set $H\subset[-1;1]$ of measure $|H|<\varepsilon$ such that
\begin{equation*}
\lambda_n(x,\omega_n,E_0)\ge c(\varepsilon)\log n\quad{\mathrm for}\quad x\in E_0\setminus H\,.
\end{equation*}
Choosing $\varepsilon<\sum_{i=0}^s(b_i-a_i)$ we get the statement of the theorem.
\end{proof}

\section{Upper estimate in the general case}

First we present a result which shows that the lower estimate expressed in Theorem 1 can be achieved by a suitably chosen set of nodes.
 To prove it we need a result of the first named author \cite{luk} about Lebesgue constants for special interpolation processes by rational functions on several intervals.

 \begin{theorem}
 \label{thmA}
 {\rm(Lukashov \cite{luk})}
 Let $ \{a_{k,n}\}_{k=1,n=s}^{\infty,\infty}\subset\{|z|<r\},\;r<1, $ be a regular matrix of inverse values of poles with respect to $ E_s, $ i.e.

 $$\sum_{k=1}^{n}\omega(a_{k,n}^{-1},[a_j,b_j],\mathbb{C}\backslash E_s)\in\mathbb{N},\quad j=0,1,\ldots, s,$$
 where $ \omega(z,\alpha,\Omega) $ denotes the harmonic measure of the set $ \alpha\subset\partial\Omega $ with respect to the domain $ \Omega $ with pole at $ z ,$ $ 1/0=\infty.$ Then the roots of the Chebyshev-Markov rational function
 \begin{equation}
\label{rational}
\tilde\omega_n(x):=\frac{\omega_n(x)}{\prod_{k=1}^n(1-a_{k,n}x)}
\end{equation}
deviating least from zero on $ E_s $ (see {\rm\cite[Theorem 3]{l2}}) are of the form {\rm(\ref{nodes})} and
\begin{equation*}
\tilde\lambda_n(\tilde\omega_n,E_s):=\left\Vert\sum_{k=1}^n\left|\frac{\tilde\omega_n(x)}{\tilde\omega_n'(x_k)(x-x_k)}\right|\right\Vert_{E_s}\le c(E_s)\log n\,.
\end{equation*}
 \end{theorem}

\begin{theorem}
\label{thm2}
For any finite partition {\rm(\ref{set})} of the interval $[-1,1]$, there exists a set of nodes {\rm(\ref{nodes})} such that
\begin{equation*}
\lambda_n(\omega_n,E_s)\le c(E_s)\log n
\end{equation*}
with a constant $c(E_s)>0$ depending on the partition {\rm(\ref{set})}.
\end{theorem}

\begin{proof} This is an easy consequence of former theorem. In the proof of \cite[Theorem 5]{l2} (which is not included in \cite{l2}, but is quite analogous to the proof of \cite[Theorem 4]{l2}, compare also \cite{l3}) it was noted that to any given complex numbers $a_{k,n}$ with $|a_{k,n}|<r\;(k=1,\dots,n)$ which are symmetric with respect to the real line  we can always find $\tilde a_{k,n}$ with  $|\tilde{a}_{m,n}|<r\; (m=n-s+1,\ldots, n)$  such that all suppositions of Theorem 2 are satisfied and $ |a_{m,n}-\tilde{a}_{m,n}|\;(m=n-s+1,\ldots, n) $ are arbitrarily small. Thus, choosing $a_{k,n}=0\;(k=1,\dots,n)$ and finding the corresponding $\tilde a_{m,n}$ with $|\tilde{a}_{m,n}|<r\;(m=n-s+1,\ldots, n) $  we obtain
\begin{equation*}
\begin{split}
c(E_s)\log n &\geq\sum_{k=1}^n\left\vert\frac{\tilde\omega_n(x)}{\tilde\omega_n'(x_k)(x-x_k)}\right\vert\\
&=\sum_{k=1}^{n}\left\vert\frac{\omega_n(x)}{\omega_n'(x_k)(x-x_k)}\right\vert\cdot\left\vert
\frac{\prod_{m=n-s+1}^{n}(1-\tilde a_{m,n}x_k)}{\prod_{m=n-s+1}^n(1-\tilde a_{m,n}x)}\right\vert\\
&\geq\frac{(1-r)^s}{2^s}\sum_{k=1}^n\left\vert\frac{\omega_n(x)}{\omega_n'(x_k)(x-x_k)}\right\vert,\qquad x\in E_s\,.
\end{split}
\end{equation*}
\end{proof}

Note that Theorem 3 gives a construction of the set of nodes with optimal order of the Lebesgue constant which is not explicit. Namely to find $\tilde a_{m,n}$ one has to solve (for $ s>2 $) a system of non-linear equations, where $\tilde a_{m,n}$ are included in hyperelliptic integrals. Then $ \omega_n(x) $ are expressed as hyperelliptic integrals and their zeros $ x_{k,n} $ are obtained as solutions of algebraic equations of $ n $th degree.  The construction of Theorem 3 can be explained in more explicit form in the special case $ E_2=[-1,a]\cup[b,1]. $

First of all we have (see, for example, \cite{luk,totik})

$$ \omega(\infty,\delta,\mathbb{C}\backslash E_2)=\frac{1}{\pi}\int_{\delta}\frac{|x-c|dx}{\sqrt{|H(x)|}}, $$
where $ \delta\subset E_2, $ $ H(x)=(x^2-1)(x-a)(x-b), $

$$c=\left.\int_a^b\frac{xdx}{\sqrt{H(x)}}\right/\int_a^b\frac{dx}{\sqrt{H(x)}},$$
and

$$\omega(\alpha,\delta,\mathbb{C}\backslash E_2)=\frac{1}{\pi}\int_{\delta}\frac{|x-c(\alpha)|\sqrt{|H(\alpha)|}dx}{\sqrt{|H(x)|}|x-\alpha|\cdot |\alpha-c(\alpha)|}, $$
where $ \alpha\in\mathbb{R}\backslash E_2, $

$$c(\alpha)=\left.\int_a^b\frac{xdx}{\sqrt{H(x)}|x-\alpha|}\right/\int_a^b\frac{dx}{\sqrt{H(x)}|x-\alpha|}.$$
Hence we can choose $ \tilde{a}_{n,n}=\alpha_n $ as a unique solution for $ \alpha\in(1,+\infty) $ of the equation

$$\frac{n-1}{\pi}\int_b^1\frac{(x-c)dx}{\sqrt{-H(x)}}+\frac{1}{\pi}\int_b^1\frac{(x-c(\alpha))\sqrt{H(\alpha)}dx}
{\sqrt{-H(x)}(\alpha-x)(\alpha-c(\alpha))}$$

$$=1+\left[\frac{n}{\pi}\int_b^1\frac{(x-c)dx}{\sqrt{-H(x)}}\right], $$
and the set of nodes $ x_k $ is defined by the equations

$$ (n-1)\int\limits_{[-1,x_k]\cap E_2}\frac{|x-c|dx}{\sqrt{-H(x)}}+\int\limits_{[-1,x_k]\cap E_2}\frac{|x-c(\alpha_n)|\sqrt{H(\alpha_n)}dx}{\sqrt{-H(x)}(\alpha_n-x)(\alpha_n-c(\alpha_n))}$$

$$=k\pi-\frac{\pi}{2},\quad k=1,\ldots,n.$$
To compute these elliptic integrals is possible with using elliptic functions (compare \cite[Theorem 2]{lukMN}).

\section{Upper estimate for two symmetric intervals}

Theorem 3 is very general, but it does not provide a concrete set of nodes for  the optimal order of Lebesgue constant. In what follows we try to handle the special case of two intervals, i.e. let
\begin{equation*}
E(a):=[-1,-a]\cup [a,1], \qquad 0<a<1\,.
\end{equation*}
With a slightly different notation for the nodes (\ref{nodes}) we have

\begin{theorem}
\label{thm3}
Consider the system of nodes on $E(a)$
\begin{equation}
\label{tnodes}
x_{\pm k}:=\pm\sqrt{\frac{1-a^2}2y_k+\frac{1+a^2}2},\quad where\quad y_k=\cos\frac{2k-1}{2n}\pi,\quad k=1,\dots,n\,,
\end{equation}
and denote $\omega_{2n}(x):=\prod_{1\le|k|\le n}(x-x_k)$. Then we have
\begin{equation*}
\lambda_{2n}(\omega_{2n},E(a))\le\frac1a\Lambda_n+\frac{1-a^2}{8a^2}\,,
\end{equation*}
where $\Lambda_n\sim\frac2\pi\log n$ is the $n^{th}$ Lebesgue constant in $[-1,1]$ for the Chebyshev nodes.
\end{theorem}

\begin{proof}
Evidently
\begin{equation*}
\omega_{2n}(x):=T_n\left(\frac{2x^2-1-a^2}{1-a^2}\right)
\end{equation*}
(where $T_n(x)=\cos n\arccos x$ is the $n$th Chebyshev polynomial). In estimating the Lebesgue function, we may assume that $a\le x\le1$, by symmetry. Introducing the notation
\begin{equation}
\label{y}
y:=\frac{2x^2-1-a^2}{1-a^2}
\end{equation}
we can easily see that
\begin{equation*}
|x-x_k|\ge
\begin{cases}
\frac{1-a^2}4\cdot|y-y_k|, & \mathrm{if}\ k=1,2,\dots,n \\
2a, & \mathrm{if}\ k=-1,-2,\dots,-n
\end{cases}
\qquad(a\le x\le1)\,.
\end{equation*}
Also,
\begin{equation*}
|\omega_{2n}'(x_k)|\ge\frac{4a}{1-a^2}\cdot|T_n'(y_k)|,\qquad k=\pm1,\pm2,\dots,\pm n\,.
\end{equation*}
Thus we obtain
\begin{equation*}
\begin{split}
\lambda_{2n}(x,\omega_{2n},E(a))&=\sum_{k=-n}^n\left\vert\frac{\omega_{2n}(x)}{\omega_{2n}'(x_k)(x-x_k)}\right\vert\\
&\le\frac{1-a^2}{8a^2}\sum_{k=1}^n\frac1{|T_n'(y_k)|}+\frac1a\sum_{k=1}^n\left\vert\frac{T_n(y)}{T_n'(y_k)(y-y_k)}\right\vert\\
&\le\frac{1-a^2}{8a^2n}\sum_{k=1}^n\sin\frac{2k-1}{2n}\pi+\frac1a\Lambda_n\\
&\le\frac{1-a^2}{8a^2}+\Lambda_n\,.
\end{split}
\end{equation*}
\end{proof}

\begin{remark}
Since $ \omega(\infty,[-1,-a],\mathbb{C}\backslash E(a))=1/2, $ a weaker estimate
$$ \lambda_{2n}(\omega_{2n},E(a))\leq C(a)\log n $$
follows from Theorem 2.
\end{remark}

The case of odd number of nodes on $E(a)$ can be settled by using the following simple

\begin{lemma}
\label{lem1}
Assume that for a system of nodes {\rm(\ref{nodes})} we have $-1<x_n<\dots<x_1<1,\quad x_k\in E_s,\;k=1,\dots,n$, and let
\begin{equation*}
\Omega_{n+1}(x):=(1+x)\omega_n(x)\qquad and\qquad\Omega_{n+2}(x):=(1-x^2)\omega_n(x)\,.
\end{equation*}
If
\begin{equation*}
|\omega_n(1)|=\Vert\omega_n\Vert_{E_s}\qquad or\qquad|\omega_n(\pm1)|=\Vert\omega_n\Vert_{E_s}\,.
\end{equation*}
then
\begin{equation*}
\lambda_{n+1}(\Omega_{n+1},E_s)\le3\lambda_n(\omega_n,E_s)+1\qquad or\qquad
\lambda_{n+2}(\Omega_{n+2},E_s)\le5\lambda_n(\omega_n,E_s)+1\,,
\end{equation*}
respectively.
\end{lemma}

\begin{proof}
We prove only the second statement; the first one is simpler. We have
\begin{equation*}
\begin{split}
\lambda_{n+2}(x,\Omega_{n+2}, E_s)=&\sum_{k=1}^n\frac{1-x^2}{1-x_k^2}\cdot\left|\frac{\omega_n(x)}{\omega_n'(x_k)(x-x_k)}\right|\\
&+\left(\frac{1+x}{|\omega_n(1)|}+\frac{1-x}{|\omega_n(-1)|}\right)\frac{|\omega_n(x)|}2\\
\le&\sum_{k=1}^n\left|\frac{\omega_n(x)}{\omega_n'(x_k)(x-x_k)}\right|+2\sum_{k=1}^n\frac{|\omega_n(x)|}{(1-x_k^2)|\omega_n'(x_k)|}+1\\
\le&\lambda_n(\omega_n,E_s)+2\sum_{x_k\ge0}\frac{|\omega_n(1)|}{(1-x_k)|\omega'_n(x_k)|}\\
&+2\sum_{x_k<0}\frac{|\omega_n(-1)|}{(1+x_k)|\omega'_n(x_k)|}+1\\
\le&5\lambda_n(\omega_n,E_s)+1.
\end{split}
\end{equation*}
\end{proof}

This lemma together with Theorem \ref{thm3} yields the following
\begin{corollary}
\label{cor1}
For each $n\ge1$ there exists a system of nodes such that
\begin{equation*}
\lambda_n(\omega_n,E(a))\le\frac3a\Lambda_{n/2}+\frac{1-a^2}{8a^2}+c
\end{equation*}
with some absolute constant $c>0$,
\end{corollary}

\section{Upper estimate for two nonsymmetric intervals}

Now we consider the case of two nonsymmetric intervals, i.e.~when
\begin{equation*}
E(a,b):=[-1,a]\cup[b,1],\qquad-1<a<b<1,\quad a+b\neq 0\,.
\end{equation*}

First of all we note that for $ \omega(\infty,[-1,a],\mathbb{C}\backslash E(a,b))=1/2\pm 1/6 $ we have, taking into account \cite[Theorem 3]{l2},\cite{apt}, as in Remark 1,
\begin{equation}
\label{star}
\lambda_{3n}(\omega_{3n},E(a,b))\leq C(a,b)\log n
\end{equation}
for $ \omega_{3n}(x)=T_n(P_3(x)), $ where $ P_3(x) $ is the cubic polynomial such that $ E(a,b)$ is the inverse image of $P_3(x)$ in $[-1,1]$. For a given $ a,-1<a<1/2, $ the existence of $ b, $ such that $  \omega(\infty,[-1,a],\mathbb{C}\backslash E(a,b))=2/3 $ follows from \cite[Theorem 3.1]{pehsch} (or from  \cite{robinson} similarly to  \cite[Lemma 1]{bog}). We want to present here an explicit construction of $ P_3(x) $ and to give a different proof of the estimate (\ref{star}). Note also that the construction and the proof do not use potential theory.
\begin{theorem}
\label{thm4}
Let $-1<a<\frac12$ and
\begin{equation}
\label{b}
b=\frac{2(1-a)(1+z)+z(a+z)}{2-a+z}
\end{equation}
where $z$ is the unique solution of the equation
\begin{equation}
\label{z}
z^3+(3-2a)z^2+(a^2-2)z-a^2+2a-2=0
\end{equation}
in the interval $(-1,a)$. Then $a<b<1$, and for the cubic polynomial
\begin{equation}
\label{cp}
\begin{split}
p(x):=& \frac{(x-z)(x-a)(x-1)}{2(1+z)(1+a)}-\frac{(x^2-1)(x-a)}{(1-z^2)(z-a)}\\
& -\frac{(x^2-1)(x-z)}{(1-a^2)(z-a)}+\frac{(x+1)(x-z)(x-a)}{2(1-z)(1-a)},
\end{split}
\end{equation}
we have
\begin{equation}
\label{p}
p(-1)=p(a)=p(b)=-1,\qquad p(z)=p(1)=1\quad\mathrm{and}\quad p'(z)=0\,.
\end{equation}
Moreover, for the set of nodes $\omega_{3n}=\prod_{k=1}^{3n}(x-x_k),$
\begin{equation}
\label{n}
(-1<)x_1<\dots<x_n<(z<)x_{n+1}<\dots<x_{2n}
\end{equation}
\begin{equation*}
<(a<b<)x_{2n+1}<\dots<x_{3n}(<1)
\end{equation*}
in $E(a,b)$ defined by
\begin{equation}
p(x_k)=p(x_{k+n})=p(x_{k+2n})=\cos\frac{2k-1}{2n}\pi,\qquad k=1,2,\dots,n
\end{equation}
we have
\begin{equation}
\label{opt}
\lambda_n(\omega_{3n},E(a,b))\le c(a)\log n\,,
\end{equation}
where $c(a)>0$ is a constant depending on $a$.
\end{theorem}
Of course, a similar statement holds when switching the roles of $a$ and $b$..

\begin{proof}
The method of proof is the same as that of Theorem \ref{thm3}, but slightly more complicated.
First we check that the equation (\ref{z}) has indeed a unique solution in the interval $(-1,a)$. Denoting the left hand side of (\ref{z}) by $q(z)$, we obtain
\begin{equation*}
q(-1)=2(1-a^2)>0\qquad\mathrm{and}\qquad q(a)=-2(1-a^2)<0\,.
\end{equation*}
This shows that $q(z)$ has a single root in each interval $(-\infty,-1), (-1,a)$ and $(a,\infty)$.

Next, we show that $b$ defined in (\ref{b}) is indeed in the interval $(a,1)$. The inequalities
\begin{equation*}
a<\frac{2(1-a)(1+z)+z(a+z)}{2-a+z}<1
\end{equation*}
are equivalent to
\begin{equation*}
z^2+2(1-a)z+a^2-4a+2>0\qquad\mathrm{and}\qquad(z-a)(1+z)<0\,.
\end{equation*}
The first inequality holds for all real $z$ (since $a<1/2$), and the second holds because of $-1<z<a$.

To check the relations (\ref{p}), only $p(b)=-1$ and $p'(z)=0$ needs explanation. The first can be seen by direct substitution of (\ref{b}) into (\ref{cp}) and using the relation (\ref{z}). The second one is obtained by differentiation and using again (\ref{z}); we omit the details.

Collecting all the above information, we may say that the cubic polynomial $p(x)$ is monotone increasing in $[-1,z]$ and in $[b,1]$, and monotone decreasing in $[z,a]$. Thus the polynomial
\begin{equation*}
\omega_{3n}(x):=T_n(p(x))\,,
\end{equation*}
of degree $3n$ has $n$ roots (\ref{n}) in each of the intervals $[-1,z], [z,a], [b,1]$.
Introducing the notations
\begin{equation*}
y=p(x)\qquad\mathrm{and}\qquad y_k=y_{k+n}=y_{k+2n}=p(x_k),\quad k=1,\dots,n\,,
\end{equation*}
we estimate the Lebesgue function of Lagrange interpolation based on these nodes:
\begin{equation*}
\begin{split}
\lambda_n(x,\omega_{3n},E(a,b))=&\sum_{k=1}^{3n}\left|\frac{T_n(y)}{p'(x_k)T_n'(y_k)(x-x_k)}\right|\\
=&\sum_{|x-x_k|\le|p'(x_k)|}+\sum_{|x-x_k|>|p'(x_k)|}\\
:=&S_1+S_2\,.
\end{split}
\end{equation*}
First we estimate $S_1$. We have by Taylor expansion
\begin{equation*}
\begin{split}
|y-y_k|=&|p(x)-p(x_k)|\\
\le&|p'(x_k)|\cdot|x-x_k|+\frac12||p''||_{[-1,1]}(x-x_k)^2\\
\le&c_1(a)|p'(x_k)|\cdot|x-x_k|\,,
\end{split}
\end{equation*}
whence
\begin{equation*}
S_1\le\frac3{c_1(a)}\sum_{k=1}^n\left|\frac{T_n(y)}{T_n'(y_k)(y-y_k)}\right|\le c_2(a)\Lambda_n\,,
\end{equation*}
since the last sum can be estimated by the Lebesgue constant associated with the Chebyshev nodes.

For estimating $S_2$, denote by $I$ an interval with midpoint $z$ such that $|p''(x)|\ge c_3(a)>0$ for $x\in I$. (Such an interval exists, since the unique zero of $p''(x)$ is in the interval $(z,b)$.) Then
\begin{equation*}
S_2\le\sum_{k=1}^{3n}\frac1{p'(x_k)^2|T_n'(y_k)|}=\sum_{x_k\notin I}+\sum_{x_k\in I}:=S_{21}+S_{22}\,.
\end{equation*}
As for the first sum we have
\begin{equation*}
S_{21}\le\frac3{c_3(a)^2}\sum_{k=1}^n\frac1{|T'_n(y_k)|}\le c_4(a)\,.
\end{equation*}

Finally, we estimate $S_{22}$. We obtain
\begin{equation*}
1-y_k=p(z)-p(x_k)\le c_4(a)(z-x_k)^2\,.
\end{equation*}
Thus 
\begin{equation*}\
\begin{split}
|p'(x_k)|=&|p'(x_k)-p'(z)|\\
=&\frac12|p''(\zeta)|\cdot|x_k-z|\\
\ge&\frac{c_3(a)}{2\sqrt{c_4(a)}}\sqrt{1-y_k},\qquad \zeta\in(z,x_k)\subset I\,.
\end{split}
\end{equation*}
Hence
\begin{equation*}
S_{22}\le c_5(a)\sum_{k=1}^n\frac1{|T_n'(y_k)|(1-y_k)}\le c_6(a)\Lambda_n\,.
\end{equation*}

So we have proved the statement for degrees of the form $3n$. For nodes of the form $3n+1$ or $3n+2$ we add to the above system of nodes the point 1, or the points $\pm1$, respectively, and apply Lemma \ref{lem1}.

To complete the proof of the theorem, we have to show that $b\neq-a$. An easy calculation yields
\begin{equation*}
1+p(-a)=2a\cdot\frac{z^2+2z+2-a^2}{(1+z)(a-z)}<\frac{4a}{a-z}<0\,,
\end{equation*}
i.e.~$p(-a)<-1=p(b).$ This shows that $b>-a$.
\end{proof}
\begin{remark}
Note that for interpolation by rational functions on $E(a,b)$ one can modify the proof of Theorem 4 to obtain similar assertion for interpolation on nonsymmetric intervals by rational functions taking $$ y=\frac{2x^2-(b+a)x-1+ab}{(b+a)x-1-ab} $$ instead of (\ref{y}).
\end{remark}

\end{document}